\documentclass{amsart}

\usepackage{ifthen}

\newtheorem{proposition}{Proposition}[section]
\newtheorem{theorem}[proposition]{Theorem}

\newtheorem{lemma}[proposition]{Lemma}
\newtheorem{prop}[proposition]{Proposition}
\newtheorem{cor}[proposition]{Corollary}

\theoremstyle{definition}

\theoremstyle{remark}
\newtheorem{remark}[proposition]{Remark}

\numberwithin{equation}{section}

\newcommand{\newword}[1]{\textbf{\textit{#1}}}

\newcommand{\reals}{\mathbb R}

\newcommand{\ep}{\epsilon}
\newcommand{\set}[1]{{\lbrace #1 \rbrace}}
\newcommand{\Set}[1]{{\big\lbrace #1 \big\rbrace}}

\newcommand{\Aff}{{\operatorname{Aff}}}
\newcommand{\Perp}{{\operatorname{Perp}}}
\renewcommand{\int}{{\operatorname{int}}}
\newcommand{\relint}{{\operatorname{relint}}}
\newcommand{\A}{{\mathcal A}}
\newcommand{\B}{{\mathcal B}}
\newcommand{\C}{{\mathcal C}}

\newcommand{\F}{{\mathcal F}}
\newcommand{\oC}{\,\overline{\mathcal{\!C}}}
\newcommand{\oE}{\,\overline{\mathcal{\!E}}}
\newcommand{\oF}{\,\overline{\!F}}
\newcommand{\oG}{\,\overline{\mathcal{\!G}}}
\newcommand{\oM}{\,\overline{\mathcal{\!M}}}
\newcommand{\oX}{\,\overline{\!X}}
\newcommand{\E}{{\mathcal E}}
\newcommand{\G}{{\mathcal G}}
\renewcommand{\L}{{\mathcal L}}
\newcommand{\M}{{\mathcal M}}
\renewcommand{\P}{{\mathcal P}}
\newcommand{\R}{{\mathcal R}}

\newcommand{\Z}{{\mathbf Z}}

\newcommand{\Supp}{\operatorname{Supp}}
\newcommand{\rank}{\operatorname{rank}}

\author{Nathan Reading}
\address{Department of Mathematics, North Carolina State University, Raleigh, NC, USA}


\thanks{The author was partially supported by NSA grant H98230-09-1-0056.}

\title{Coarsening polyhedral complexes}

\begin{document}

\begin{abstract}
Given a pure, full-dimensional, locally strongly connected polyhedral complex~$\C$, we characterize, by a local codimension-$2$ condition, polyhedral complexes that coarsen~$\C$.
The proof of the characterization draws upon a general shortcut for showing that a collection of polyhedra is a polyhedral complex and upon a property of hyperplane arrangements which is equivalent, for Coxeter arrangements, to Tits' solution to the Word Problem.
The motivating special case, the case where~$\C$ is a complete fan, generalizes a result of Morton, Pachter, Shiu, Sturmfels, and Wienand that equates convex rank tests with semigraphoids. 
We also prove oriented matroid versions of our results, obtaining, as a byproduct, an oriented matroid version of Tietze's convexity theorem.
\end{abstract}

\maketitle

\section{Summary of results}\label{intro}
The purpose of this paper is to characterize the polyhedral complexes~$\C'$ that coarsen a given polyhedral complex~$\C$.
A \newword{polyhedron} is an intersection of finitely many closed halfspaces.
A \newword{polyhedral complex} is a finite, nonempty collection~$\C$ of polyhedra such that (1) if $F\in\C$ and~$G$ is a face of~$F$, then $G\in\C$, and (2) if~$F$ and~$G$ are in~$\C$, then $F\cap G$ is a face of~$F$ and a face of~$G$.
The polyhedra in~$\C$ are called the \newword{faces} of~$\C$.
A \newword{fan} is (the set of nonempty faces of) a polyhedral complex all of whose nonempty faces contain the origin.
Details on polyhedra, polyhedral complexes and fans can be found, for example, in~\cite{Ziegler}.
We typically shorten ``polyhedral complex'' to ``complex.''
The \newword{dimension} of $\C$ is the maximum of the dimensions of its faces, and $\C$ is called \newword{pure} if all of its maximal faces have the same dimension.
The \newword{support} $\Supp(\C)$ of a collection~$\C$ of polyhedra is the union of the polyhedra in the collection.
A complex is \newword{complete} if its support is the entire ambient space.
A complex~$\C'$ \newword{coarsens} a complex~$\C$ if~$\C'$ and~$\C$ have the same support and if each face of~$\C'$ is a union of faces of~$\C$.

Let~$\C$ be a pure, $n$-dimensional polyhedral complex in $\reals^n$.
The \newword{adjacency graph}~$\G$ of~$\C$ is the graph whose vertices are the full-dimensional faces of~$\C$ and whose edges are the pairs of adjacent full-dimensional faces (pairs of full-dimensional faces whose intersection is a codimension-$1$ face).
For each face $F$ of $\C$, the \newword{local adjacency graph of $\C$ at $F$} is the subgraph of $\G$ induced by maximal faces of $\C$ containing $F$.
The complex $\C$ is \newword{locally strongly connected} if, for every face $F$, the local adjacency graph of $\C$ at $F$ is connected.
Given a complex~$\C'$ coarsening~$\C$, define the \newword{edge set of~$\C'$} to be the set of edges~$M$---$N$ in~$\G$ such that~$M$ and~$N$ are contained in the same face of~$\C'$.
A complex~$\C'$ coarsening~$\C$ is uniquely determined by its edge set and vice versa.  
Thus to characterize complexes coarsening~$\C$, we give a necessary and sufficient local condition for a set~$\E$ of edges of~$\G$ to be the edge set of a complex that coarsens~$\C$.

Let $F$ be a codimension-$2$ face of $\C$ (often called a \newword{ridge} of $\C$) and let $\P$ be the local adjacency graph of $\C$ at $F$.
Let $\Aff(F)$ denote the \newword{affine hull} of~$F$, the intersection of all affine hyperplanes containing~$F$.
This is an affine subspace of codimension $2$.
Let $\Perp(F)$ be the unique linear subspace orthogonal to $\Aff(F)$.
This $2$-dimensional plane is the orthogonal complement of the linear subspace $\Aff(F)-p$, where $p$ is any point in $\Aff(F)$.
Choose a point~$x$ in the relative interior of~$F$.
For each face~$M$ in~$\P$, define a cone $\set{v\in \Perp(F):\,\exists\,\ep>0\mbox{ with }x+\ep v\in M}$.
The cones that arise in this way are the maximal cones of a fan $\C|_F$ in the plane $\Perp(F)$.
The adjacency graph of $\C|_F$ is~$\P$.
A set~$\E$ of edges in~$\G$ has the \newword{ridge property} if, for every codimension-$2$ face~$F$ of $\C$, the restriction of~$\E$ to~$\P$ is the edge set of a fan in $\Perp(F)$ coarsening $\C|_F$.
We will prove the following theorem.

\begin{theorem}\label{coarsen}
Let~$\C$ be a pure, full-dimensional, locally strongly connected polyhedral complex and let~$\G$ be the adjacency graph on maximal faces of~$\C$.
Then a subset~$\E$ of the edges of~$\G$ is the edge set of a complex coarsening~$\C$ if and only if~$\E$ has the ridge property.
\end{theorem}

One key ingredient in the proof of Theorem~\ref{coarsen} is a shortcut for proving that a collection of polyhedra is a polyhedral complex.
Given a collection~$\M$ of polyhedra, for each integer $k\ge -1$, let $\bigcap_k(\M)$ be the union of all intersections $M\cap N$ such that $M,N\in\M$ and $\dim(M\cap N)\le k$.
(By convention, the empty set has dimension~$-1$.)
For any $x\in\reals^n$ and $\delta>0$, let $\B_\delta(x)$ be the open ball $\set{y\in\reals^n:|x-y|<\delta}$.
Although the proof of Theorem~\ref{coarsen} only needs a special case of the following theorem (the case where $\M$ consists of $n$-dimensional polyhedra and $k=n-2$), it is not significantly harder to prove the more general statement.

\begin{theorem}\label{shortcut}
Fix $k\ge -1$ and let~$\M$ be a finite collection of polyhedra in $\reals^n$, each of dimension greater than~$k$.
Suppose:
\begin{enumerate}
\item[(i) ] For all $x\in\Supp(\M)$, there exists $\ep>0$ such that, for all $\delta$ with $\ep>\delta>0$, the set $\left[\Supp(\M)\setminus\bigcap_k(\M)\right]\cap\B_\delta(x)$ is path connected; and
\item[(ii) ] $M\cap N$ is a face of~$M$ and of~$N$ for all $M,N\in\M$ with $\dim(M\cap N)>k$.
\end{enumerate}
Then the collection of all polyhedra in~$\M$ and their faces is a polyhedral complex.
\end{theorem}

Let $\C(\A)$ be the complete polyhedral complex determined by an affine hyperplane arrangement~$\A$. 
In Section~\ref{edge sec}, we prove Theorem~\ref{coarsen}, beginning with the special case where~$\C$ is a pure, full-dimensional, locally strongly connected subcomplex of $\C(\A)$.
The proof relies on the observation, discussed in Section~\ref{path sec}, that every hyperplane arrangement has a property that we call path convexity.
When~$\A$ is a Coxeter arrangement, the statement that~$\A$ is path convex is exactly the statement of Tits' solution \cite[Th\'{e}or\`{e}me~3]{Tits} to the Word Problem for the corresponding Coxeter group.
Essentially equivalent observations have been made in various other settings \cite{CorMor,Deligne,Salvetti}.

Key parts of the proof of Theorem~\ref{coarsen} are shared by proofs of the less general results \cite[Proposition~5.2]{con_app} and \cite[Theorem~9]{ranktests}.
The former, for a broad class of central hyperplane arrangements~$\A$, uses a stronger condition than the ridge property, arising from the lattice theory of the weak order, to show that certain sets of edges of the adjacency graph $\G(\A)$ are the edge sets of fans coarsening the fan $\C(\A)$.
The latter establishes Theorem~\ref{coarsen} in the case $\C=\C(\A)$, where~$\A$ is the Coxeter arrangement for the symmetric group (i.e.\ the braid arrangement).
For this~$\A$, the complex $\C(\A)$ is the normal fan of the permutohedron, and fans coarsening $\C(\A)$ are interpreted in the language of nonparametric statistics as \newword{convex rank tests}.
The edge sets of fans coarsening $\C(\A)$ are characterized by the \newword{square axiom} and the \newword{hexagon axiom}.
Furthermore, edge sets satisfying the square axiom and the hexagon axiom are identified with certain conditional independence structures known as \newword{semigraphoids}. 

Theorem~\ref{coarsen} in particular solves the problem, posed in \cite[Section~1]{ranktests},  of characterizing the edge sets of coarsenings of $\C(\A)$ for arbitrary Coxeter arrangements $\A$.
Indeed, the theorem is particularly simply stated when $\C=\C(\A)$ for any central hyperplane arrangement $\A$, as we now explain.

Let $\Z$ be a zonotope.
A set~$\E$ of edges of~$\Z$ has the \newword{polygon property} if, for every $2k$-gonal face $P$ of $\Z$, whenever~$\E$ contains any $k-1$ consecutive edges of $P$, then~$\E$ also contains the opposite $k-1$ consecutive edges of~$P$.
The polygon property on the usual permutohedron coincides with the square and hexagon axioms from~\cite{ranktests}.
When $\A$ is the central hyperplane arrangement dual to $\Z$ and $\C$ is $\C(\A)$ (the normal fan to~$\Z$), then the $2$-dimensional fans $\C_F$ are defined by an arrangement of lines through the origin in $\Perp(F)$.
Thus the ridge property reduces to the polygon property, and we have the following corollary to Theorem~\ref{coarsen}.
\begin{cor}\label{coarsen zonotope}
Let $\Z$ be a zonotope and let $\F$ be the normal fan of $\Z$.
Then a set $\E$ of edges of $\Z$ is the edge set of a fan coarsening $\F$ if and only if $\E$ has the polygon property.
\end{cor}
When $\A$ is a non-central hyperplane arrangement, the analogous polygon property still characterizes coarsenings of $\C(A)$ for the same reason.

The proof of Theorem~\ref{coarsen} also provides a local condition (Theorem~\ref{Tietze lite}) for an interior-connected union of polyhedra to be convex, which is a special case of Tietze's convexity theorem~\cite{Tietze}.
(See \cite[Part~IV.C]{Valentine}.)
In Section~\ref{OM sec}, we extend all of our results to the context of oriented matroids, proving, in particular, an oriented matroid version (Theorem~\ref{Tietze OM}) of Tietze's convexity theorem.

\section{Polyhedral complexes}\label{complex sec}
In this section, we prove Theorem~\ref{shortcut}.
We begin by establishing a well-known, easier result.
(See, for example, \cite[Lemma~14]{ranktests} or \cite[Lemma~3.2]{con_app}.)

\begin{lemma}\label{longer shortcut}
Let~$\M$ be a finite collection of polyhedra and let~$\C$ be the collection consisting of all polyhedra in~$\M$ and their faces.
Suppose $M\cap N$ is a face of~$M$ and of~$N$ for all $M,N\in\M$.
Then~$\C$ is a polyhedral complex.
\end{lemma}
\begin{proof}
Let~$F$, $G$,~$M$, and $N$ be polyhedra in~$\C$ such that~$F$ is a face of~$M$ and $G$ is a face of $N$.
We claim that $F\cap G$ is a face of $M\cap N$.
If $F=M$ and $G=N$, then the assertion is trivial, so without loss of generality,~$F$ is a proper face of~$M$.
Let $H$ be a hyperplane such that $H\cap M$ is the face~$F$ of~$M$.
Then also $H\cap(M\cap N)=F\cap N$ is a face of $M\cap N$.
If $G=N$, then $F\cap G$ is a face of $M\cap N$, and if not, we argue similarly that $M\cap G$ is a face of $M\cap N$.
Thus $F\cap G=(F\cap N)\cap(M\cap G)$ is a face of $M\cap N$, and we have proven the claim in either case.
Now, since~$F$ and $F\cap G$ are faces of $M\cap N$, $F\cap G$ is a face of~$F$.
Symmetrically, $F\cap G$ is a face of $G$.
\end{proof}

\begin{proof}[Proof of Theorem~\ref{shortcut}] 
We will verify the hypotheses of Lemma~\ref{longer shortcut}.
Let~$M$ and~$N$ be distinct polyhedra in~$\M$, let~$F$ be the polyhedron $M\cap N$, and let $d$ be the dimension of~$F$.
If $F=\emptyset$, then we are done.
Otherwise, there exists a point $x\in F$ such that $x$ is not contained in any polyhedron in $\C$ of dimension strictly less than~$d$.
By hypothesis (i), there exists $\ep>0$ such that, for all $\delta$ with $\ep>\delta>0$, the set $(\Supp(\M)\setminus\bigcap_k(\M))\cap\B_\delta(x)$ is path connected.
Every face of~$\C$ not containing~$x$ is some positive distance from~$x$ and there are finitely many faces of~$\C$.
Thus there exists $\delta$ with $\ep>\delta>0$ such that every face of~$\C$ intersecting $\B_\delta(x)$ actually contains~$x$.

Now~$x$ is in~$M$ and in~$N$, so $\B_\delta(x)$ intersects the relative interiors of~$M$ and~$N$.
Let $y\in(\B_\delta(x)\cap\relint(M))\setminus\bigcap_k(\M)$ and $z\in(\B_\delta(x)\cap\relint(N))\setminus\bigcap_k(\M)$.
Let $\alpha:[0,1]\to(\Supp(\M)\setminus\bigcap_k(\M))\cap\B_\delta(x)$ be a path from~$y$ to $z$.
We will use $\alpha$ to construct a sequence $M=M_0,\ldots,M_j=N$ of polyhedra in $\M$ such that, for each $i=1,\ldots,j$, the intersection $M_{i-1}\cap M_i$ is of dimension greater than~$k$.
The set~$M$ is closed, so $\alpha^{-1}(M)$ is a closed subset of $[0,1]$.
If $t_1$ is the maximum of the set $\alpha^{-1}(M)$, then $\alpha(t_1)$ is in~$M$ and in some $M_1\in\M$.
But $\alpha(t_1)\in \Supp(\M)\setminus\bigcap_k(\M)$, so $M\cap M_1$ has dimension greater than~$k$.
If $M_1\neq N$, then repeat the construction to find $M_2\in\M$ such that $\dim(M_1\cap M_2)>k$, and continue until $M_j=N$.

We now show that for any sequence $M_0,\ldots,M_j$ of polyhedra in $\M$ such that $M_{i-1}\cap M_i$ is of dimension greater than~$k$ for each $i=1,\ldots,j$, the intersection $M_0\cap\cdots\cap M_j$ is a face of $M_j$.
We argue by induction on~$j$, the case $j=0$ being trivial.
If $j>0$, then by induction $M_0\cap\cdots\cap M_{j-1}$ is a face~$G$ of $M_{j-1}$.
Now $G'=M_{j-1}\cap M_j$ is a face of $M_{j-1}$ and of $M_j$ by \hyphenation{hypo-thesis} hypothesis (ii).
Thus $M_0\cap\cdots\cap M_j$ is a face of $M_{j-1}$ because it is the intersection of two faces,~$G$ and $G'$, of $M_{j-1}$.
But then $M_0\cap\cdots\cap M_j$ is a face of $G'$, and thus a face of $M_j$.

We have shown that $F'=M_0\cap\cdots\cap M_j$ is a face of~$N$.
Since $\alpha$ is contained in $\B_\delta(x)$, each $M_i$ intersects $\B_\delta(x)$, so by the definition of $\delta$, each $M_i$ contains~$x$.
Thus $F'$ contains~$x$, so by the definition of~$x$, $F'$ is a face of~$\C$ of dimension at least~$d$.
However, $F'$ is contained in the $d$-dimensional polyhedron $F=M\cap N$, so~$F'$ has dimension~$d$.
Let $H$ be a hyperplane such that $N\cap H$ is $F'$.
Since $F'\subseteq F$ and both are~$d$-dimensional polyhedra, we have $F\subseteq H$.
Thus $F\subseteq (N\cap H)=F'$.
Therefore $F'=F$, so that~$F$ is a face of~$N$.
By symmetry,~$F$ is a face of~$M$.
\end{proof}

\section{Path convexity}\label{path sec}
In this section, we show that every hyperplane arrangement has a property that we call path convexity.
This fact will be crucial in the proof of Theorem~\ref{coarsen}.

A \newword{hyperplane arrangement} in $\reals^n$ is a finite collection~$\A$ of affine hyperplanes.
The closures of the connected components of $\reals^n\setminus\left(\bigcup_{H\in\A}H\right)$ are called \newword{regions}.
The regions are the maximal faces of a complete polyhedral complex $\C(\A)$.
Let $\G(\A)$ be the adjacency graph of the complex $\C(\A)$.

Let $Q,R\in\R(\A)$.
A \newword{path} in $\G(\A)$ from~$Q$ to~$R$ is a sequence $R_0,R_1,\ldots,R_k$ of regions with $Q=R_0$ and $R=R_k$, such that $R_{i-1}$---$R_i$ is an edge in $\G(\A)$ for each $i$ from $1$ to $k$.
The length of a path $R_0,R_1,\ldots,R_k$ is~$k$, one less than the number of entries in the sequence.
A \newword{braid move} on a path alters the path by deleting an adjacent subsequence $Q_0,\ldots,Q_m$ from the path and replacing it with a sequence $Q'_0,\ldots,Q'_m$ such that $Q_0=Q_0'$, $Q_m=Q_m'$ and the cycle $Q_0,Q_1,\ldots,Q_mQ'_{m-1},\ldots,Q'_0$ is a polygon in~$\G(\A)$.
A braid move does not change the length of the path.
A \newword{nil move} on paths alters a path by replacing an adjacent subsequence $Q_0,Q_1,Q_2$ such that $Q_0=Q_2$ by the singleton sequence $Q_0$.

We say that a path $\gamma$ is \newword{reduced} if it has minimal length among all paths from~$Q$ to~$R$.
The arrangement~$\A$ is \newword{path convex} if, for every pair $Q,R$ of regions in~$\A$, every path $\gamma$ from~$Q$ to~$R$, and every reduced path~$\rho$ from~$Q$ to~$R$, the path $\gamma$ can be transformed, by a sequence of braid moves and nil moves, to the path~$\rho$.
The appropriateness of the term ``convex'' in this definition will become apparent in the proof of Theorem~\ref{coarsen}, particularly in Lemma~\ref{weak poly convex}.

\begin{theorem}\label{path conv}
Every hyperplane arrangement is path convex.
\end{theorem}

A slightly weaker statement for oriented matroids is \cite[Proposition~4.4.6]{OrientedMatroids}.
We now prepare to prove Theorem~\ref{path conv}.
Given $Q,R\in\R(\A)$, let $S(Q,R)$ be the set of hyperplanes of~$\A$ that separate~$Q$ from~$R$.
The following lemma is well known.

\begin{lemma}\label{red crit}
A path from~$Q$ to~$R$ is reduced if and only if its length is $|S(Q,R)|$.
\end{lemma}
\begin{proof}
Moving from one region to an adjacent region, one crosses exactly one hyperplane of~$\A$.
Thus a path from~$Q$ to~$R$ has length at least $|S(Q,R)|$.
If~$x$ is a generic point in the interior of~$Q$ and~$y$ is a generic point in the interior of~$R$, then the line segment $\overline{xy}$ intersects each hyperplane in $S(Q,R)$ exactly once, intersects no two hyperplanes in $S(Q,R)$ in the same point, and intersects no hyperplane of $\A\setminus S(Q,R)$.
Thus $\overline{xy}$ defines a path of length $|S(Q,R)|$ from~$Q$ to~$R$. 
\end{proof}

The arrangement~$\A$ is \newword{reduced-path connected} if, for every pair $Q,R$ of regions in~$\A$ and every pair $\gamma,\rho$ of reduced paths from~$Q$ to~$R$, the path $\gamma$ can be transformed, by a sequence of braid moves, to the path~$\rho$.

\begin{lemma}\label{red path enough}
If~$\A$ is reduced-path connected, then~$\A$ is path convex.
\end{lemma}
\begin{proof}
Suppose~$\A$ is reduced-path connected.
Let $\gamma=(R_0,R_1,\ldots,R_m)$ be any path from~$Q$ to~$R$ and let~$\rho$ be any reduced path from~$Q$ to~$R$.

If $\gamma$ is not reduced, then Lemma~\ref{red crit} says that $m>|S(P,Q)|$.
Thus there exists a smallest positive integer~$k$ such that $k>|S(R_0,R_k)|$.
Then $R_0,R_1,\ldots,R_{k-1}$ is a reduced path and $|S(R_0,R_k)|=k-2$.
By Lemma~\ref{red crit}, there is a reduced path $R'_0,R'_1,\ldots,R'_{k-2}$ from $R_0$ to $R_k$, and thus the path $R'_0,R'_1,\ldots,R'_{k-2},R_{k-1}$ is reduced. 
(Notice that an \textbf{unprimed} $R_{k-1}$ is the last region in this path.)
Since~$\A$ is reduced-path connected, there is a sequence of braid moves that transforms $R_0,R_1,\ldots,R_{k-1}$ to $R'_0,R'_1,\ldots,R'_{k-2},R_{k-1}$.
The same braid moves transform the path $\gamma$ to $R'_0,R'_1,\ldots,R'_{k-2},R_{k-1},R_k,\ldots,R_m$.
But $R'_{k-2}=R_k$, so a nil move can be applied to $R'_0,R'_1,\ldots,R'_{k-2},R_{k-1},R_k,\ldots,R_m$, replacing $R'_{k-2},R_{k-1},R_k$ with~$R_k$.

Repeating the process, we transform $\gamma$ to a reduced path $\gamma'$ by a sequence of braid moves and nil moves.
By the reduced-path connectedness of~$\A$, $\gamma'$ can be transformed to~$\rho$ by a sequence of braid moves.
\end{proof}

Lemma~\ref{red path enough} reduces Theorem~\ref{path conv} to the following theorem.

\begin{theorem}\label{red path prop}
Every hyperplane arrangement is reduced-path connected.
\end{theorem}

Theorem~\ref{red path prop} was proved by Deligne \cite[Proposition~1.12]{Deligne} for simplicial hyperplane arrangements, by Salvetti \cite[Lemma~11]{Salvetti}, and also by Cordovil and Moreira \cite[Theorem~2.4]{CorMor} for oriented matroids.
For the sake of completeness, we give a short proof which is similar to the argument given in \cite{CorMor,Salvetti}.

\begin{proof}[Proof of Theorem~\ref{red path prop}]
Let $\gamma=(Q_0,Q_1,\ldots,Q_k)$ and $\rho=(R_0,R_1,\ldots,R_k)$ be reduced paths with $Q_0=R_0$ and $Q_k=R_k$.
We will show that $\gamma$ and~$\rho$ are related by a sequence of braid moves.

Let $F_0,F_1,\ldots,F_m$ be a sequence of facets (maximal proper faces) of $Q_0$, chosen to minimize $m$ subject to the following requirements: (1) that $F_0=Q_0\cap Q_1$, (2) that $F_m=Q_0\cap R_1$, (3) that $F_{i-1}\cap F_i$ has codimension $2$ for each $i=1,\ldots,m$, and (4) that, for each $i=0,1,\ldots,m$, the hyperplane $H_i$ containing $F_i$ is in the set $S(Q_0,Q_k)$.
We will show that such a sequence $F_0,F_1,\ldots,F_m$ exists.
For each $i=1,\ldots, k$, choose a point $x_i\in Q_{i-1}\cap Q_i$ and a point $y_i\in R_{i-1}\cap R_i$ and concatenate the segments $\overline{x_1x_2}$, \ldots, $\overline{x_{k-1}x_k}$, $\overline{x_ky_k}$, and $\overline{y_ky_{k-1}}$, \ldots, $\overline{y_2y_1}$ to construct a continuous curve $\alpha:[0,1]\to\reals^n$ that begins in $Q_0\cap Q_1$, passes through $Q_1,\ldots,Q_k,R_{k-1},\ldots,R_1$, ending in $Q_0\cap R_1$.
Choose a point~$x$ in the relative interior of $Q_0$.
Define a continuous curve $\beta$ in the boundary of $Q_0$ by taking $\beta(t)$ to be the unique point on the boundary of $Q_0$ and on the line segment with endpoints~$x$ and $\alpha(t)$.
Let~$U$ be the union of all lines that contain~$x$ and that intersect a face of $Q_0$ of codimension $3$ or greater.
Since~$U$ has codimension $2$, for generic choices of the $x_i$ and $y_i$, the path $\alpha$ avoids~$U$.
Thus $\beta$ avoids faces of $Q_0$ of codimension $3$ or greater, so $\beta$ defines a sequence $F_0,F_1,\ldots,F_m$ of facets of $Q_0$ satisfying requirements (1), (2), and (3).
To see that the sequence satisfies requirement (4), note that each $H_i$ intersects a line segment connecting a point in the interior of some $Q_j$ or $R_j$, for $1\le j\le k$, to the point $x\in\int(Q_0)$.
Thus $H_i\in S(Q_0,Q_j)$ or $H\in S(Q_0,R_j)$ for some $1\le j\le k$.
But since $\gamma$ and~$\rho$ are reduced paths, we have $S(Q_0,Q_1)\subset S(Q_0,Q_2)\subset\cdots\subset S(Q_0,Q_k)$, and $S(Q_0,R_1)\subset S(Q_0,R_2)\subset\cdots\subset S(Q_0,R_k)$, so $H_i\in S(Q_0,Q_k)$.

We now argue by induction on~$k$ and on $m$.
If $k=0$, then $Q_0=Q_k$ and the assertion is trivial.
Now suppose $k>0$.
If $m=0$, then $F_0=F_m$, so $Q_1=R_1$.
By induction on~$k$, there is a sequence of braid moves relating $Q_1,\ldots,Q_k$ to $R_1,\ldots,R_k$.
This same sequence of braid moves relates $\gamma$ to~$\rho$.

Now suppose $m>0$ as well.
Let $\A'$ be the set of hyperplanes in~$\A$ containing the codimension-$2$ face $F_{m-1}\cap F_m$.
Any region $T$ with $\set{H_{m-1},H_m}\subseteq S(Q_0,T)$ has $\A'\subseteq S(Q_0,T)$.
In particular, $\A'\subseteq S(Q_0,Q_k)$.
Let~$x$ be a point in the relative interior of $Q_0$ and let $p$ be a point in the relative interior of $F_{m-1}\cap F_m$.
For small enough $\ep>0$, the point $p+\ep(p-x)$ is in a region $T$ with $S(Q_0,T)=\A'$.
Let $\mu$ be a reduced path from $T$ to $Q_k$.
Then $\mu$ has length $|S(Q_0,Q_k)|-|\A'|$ because $S(T,Q_k)=S(Q_0,Q_k)\setminus\A'$.
There are two reduced paths from $Q_0$ to $T$, related by a braid move involving the polygon dual to $F_{m-1}\cap F_m$.
Concatenating these paths with $\mu$, we obtain a reduced path $\rho'$ from $Q_0$ to $Q_k$ starting with the regions $Q_0,R_1$, and a reduced path $\gamma'$ from $Q_0$ to $Q_k$ starting with $Q_0$ and then continuing to the region which shares the facet $F_{m-1}$ with $Q_0$.
By induction on $m$, the paths $\gamma$ and $\gamma'$ are related by a sequence of braid moves.
By construction, $\gamma'$ and $\rho'$ are related by a single braid move.
Let $\rho'=R'_0,R'_1,\ldots,R'_k$, so that $R'_0=R_0=Q_0$, $R'_1=R_1$, and $R_k'=R_k=Q_k$.
By induction on~$k$, $R'_1,\ldots,R'_k$ and $R_1,\ldots,R_k$ are related by a sequence of braid moves, so $\rho'$ and~$\rho$ are related by the same sequence of braid moves.
We have found a sequence of braid moves relating $\gamma$ and~$\rho$.
\end{proof}

\section{Edge sets of coarsenings}\label{edge sec}
In this section, we prove Theorem~\ref{coarsen}.
One direction of the theorem is easy.
Indeed, suppose that~$\C'$ is a complex coarsening~$\C$.
If the edge set of~$\C'$ fails the ridge property at some codimension-$2$ face~$F$ of~$\C$, then we reach a contradiction to the supposition that $\C'$ is a complex:
Either some maximal face of~$\C'$ is not convex or there is a pair of maximal faces $C$ and $D$ of~$\C'$, each having~$F$ in their boundary, such that $C\cap D$ is not a face of~$C$.
This contradiction proves the ``only if'' assertion of Theorem~\ref{coarsen}.

Let~$\A$ be a hyperplane arrangement and continue the notation of Section~\ref{path sec}.
We first prove Theorem~\ref{coarsen} in the special case where $\C$ is a pure, full-dimensional, locally strongly connected subcomplex of $\C(\A)$.
Recall that in Section~\ref{intro} we defined the polygon property for a set $\E$ of edges of $\G(\A)$ when $\A$ is a central arrangement and pointed out that the polygon property is equivalent to the ridge property in this case.
We now generalize the polygon property by allowing $\A$ to be non-central and by allowing $\C$ to be a subcomplex of $\C(\A)$.

Let $F$ be a codimension-$2$ face of~$\C(\A)$.
The \newword{polygon} in~$\G(\A)$ associated to~$F$ is the cycle~$\P$ in~$\G(\A)$ consisting of all of the full-dimensional faces of~$\C(\A)$ containing~$F$.
Every polygon in $\G(\A)$ is a $2k$-gon for some $k\ge 2$.
Let $\C$ be a pure, full-dimensional, locally strongly connected subcomplex of $\C(\A)$ with adjacency graph $\G$.
A set~$\E$ of edges of~$\G$ has the \newword{polygon property} if the following condition holds for every $2k$-gon $\P$ in $\G(\A)$:
If~$\E$ contains any $k-1$ consecutive edges of $\P$, then either the opposite $k$ vertices of $\P$ are not vertices of $\G$ or~$\E$ also contains the opposite $k-1$ consecutive edges of~$\P$.
(It is important to keep in mind that the polygon property for $\E$ is defined in terms of polygons in $\G(\A)$, not in $\G$.)
It is immediate that the ridge property is equivalent to the polygon property in the case where $\C$ is a pure, full-dimensional subcomplex of $\C(\A)$.

A set~$\E$ of edges of $\G$ has the \newword{weak polygon property} if the following condition holds for each $2k$-gon~$\P$ in $\G(\A)$:
If~$\E$ contains any~$k$ consecutive edges of~$\P$, then~$\E$ contains all of the edges of~$\P$.
The polygon property implies the weak polygon property.

A \newword{pre-complex}~$\C'$ is a collection of polyhedra such that if $F\in\C'$ and~$G$ is a face of~$F$, then $G\in\C'$.
A pre-complex~$\C'$ \newword{coarsens} a complex~$\C$ if~$\C'$ and~$\C$ have the same support and if each face of~$\C'$ is a union of faces of~$\C$.
(This matches the definition given earlier when $\C'$ was assumed to be a complex.)
We first show that the weak polygon property characterizes pre-complexes coarsening a full-dimensional subcomplex~$\C$ of~$\C(\A)$.

\begin{lemma}\label{weak poly convex}
Let $\C$ be a full-dimensional subcomplex of $\C(\A)$ and let~$\E$ be a set of edges in $\G$ with the weak polygon property.
Then~$\E$ is the edge set of a pre-complex $\C'$ coarsening~$\C$.
The maximal faces of $\C'$ are full-dimensional and have pairwise disjoint interiors.
\end{lemma}
\begin{proof}
We will think of~$\E$ not only as a set of edges of $\G$, but as a graph in its own right, whose vertices are the full-dimensional faces of~$\C$.
Each connected component of~$\E$ is in particular a set of regions.
Consider the union of this set of regions.
We must show that each such union is a polyhedron.
In fact, we need only show convexity; the fact that the union is a polyhedron will follow.

Let~$M$ be a such a union.
Choose points $x,y\in M$.  
We will show that the line segment $\overline{xy}$ is contained in~$M$.
If there is a region~$R$ of~$\A$ with $\set{x,y}\subseteq R\subseteq M$, then $\overline{xy}\subseteq R\subseteq M$, so suppose~$x$ and~$y$ are not contained in the same region.
Let~$Q$ be a region with $x\in Q\subseteq M$ and let~$R$ be a region with $y\in R\subseteq M$.
Suppose for the moment that $x\in\int(Q)$ and $y\in\int(R)$ and that the line segment $\overline{xy}$ does not intersect any face of $\C(\A)$ of codimension greater than~$1$.
Then, as in the proof of Lemma~\ref{red crit}, $\overline{xy}$ defines a path~$\rho$ of length $|S(Q,R)|$ from~$Q$ to~$R$. 
By Lemma~\ref{red crit},~$\rho$ is a reduced path.
On the other hand, since~$Q$ and~$R$ are both contained in~$M$, there is a path $\gamma=(R_0,R_1,\ldots,R_k)$ from~$Q$ to~$R$ which is not only a path in $\G(\A)$, but also a path in~$\E$.
By the path convexity of~$\A$, the path $\gamma$ can be converted to the path~$\rho$ by a sequence of braid moves and nil moves.
Trivially, each nil move applied to $\gamma$ produces a new path in~$\E$.
Furthermore, the weak polygon property of~$\E$ implies that, when a braid move is performed on $\gamma$, the new path is also a path in~$\E$.
We conclude that~$\rho$ is a path in~$\E$.
In particular, each region in~$\rho$ is contained in~$M$, so $\overline{xy}\subseteq M$.

If $x\not\in\int(Q)$, if $y\not\in\int(R)$ and/or if $\overline{xy}$ intersects a face of $\C(\A)$ of codimension greater than~$1$, then there exist points $x'\in\int(Q)$ and $y'\in\int(R)$, arbitrarily close to~$x$ and~$y$ respectively, such that $\overline{x'y'}$ does not intersect any face of $\C(\A)$ of codimension greater than~$1$.
Therefore each point on $\overline{xy}$ is arbitrarily close to a point which we have proven to be in~$M$.
Since~$M$ is a union of finitely many closed polyhedra, it is closed, so $\overline{xy}\subseteq M$.

The collection of all such polyhedra~$M$ and their nonempty faces is a pre-complex~$\C'$ coarsening $\C(\A)$.
By construction, the maximal faces of $\C'$ are full-dimensional and have pairwise disjoint interiors.
It remains to show that~$\E$ is indeed the edge set of~$\C'$.
By construction,~$\E$ is contained in the edge set of~$\C'$.
Suppose $\set{Q,R}$ is a pair of regions contained in the same maximal face of~$\C'$.
Then the path $Q,R$ is a reduced path~$\rho$ (in $\G(\A)$) from~$Q$ to~$R$.
Since~$Q$ and~$R$ are contained in the same maximal face of~$\C'$, there exists a path $\gamma$ in~$\E$ from~$Q$ to~$R$.
Arguing as above, the path~$\rho$ is also a path in~$\E$, or in other words, $Q$---$R$ is an edge in~$\E$.
We have shown that~$\E$ is the edge set of~$\C'$.
\end{proof}

\begin{lemma}\label{codim 1 intersections}
Let $\C$ be a full-dimensional subcomplex of $\C(\A)$.
Let~$\E$ be a set of edges in $\G$ with the polygon property, so that~$\E$ is the edge set of a pre-complex~$\C'$ coarsening~$\C$, by Lemma~\ref{weak poly convex}.
If~$M$ and~$N$ are maximal faces of~$\C'$ and $M\cap N$ has codimension~$1$, then $M\cap N$ is a face of~$M$ and a face of~$N$.
\end{lemma}
\begin{proof}
Suppose~$M$ and~$N$ are maximal faces of~$\C'$ and $M\cap N$ has codimension~$1$.
Since~$M$ and~$N$ are unions of faces of $\C(\A)$, their intersection is a union of codimension-$1$ faces of $\C(\A)$.
Since~$M$ and~$N$ are convex and have disjoint interiors, $M\cap N$ is contained in some face~$F$ of~$M$ of codimension~$1$.
The face~$F$ is also a union of codimension-$1$ faces of $\C(\A)$.
We now prove the following claim:
If~$G$ is a codimension-$1$ face of $\C(\A)$ contained in $M\cap N$ and $G'$ is a codimension-$1$ face of $\C(\A)$ contained in~$F$ such that $G\cap G'$ has codimension~$2$, then $G'\subseteq (M\cap N)$.

To prove the claim, let~$\P$ be the polygon consisting of regions containing the codimension-$2$ face $G\cap G'$ of $\C(\A)$.
Since~$G$ and $G'$ are both in the face~$F$ of~$M$, the polygon~$\P$ contains a region~$R$ with $G\subset R\subset M$ and a region $R'$ with $G'\subset R'\subset M$.
Since~$G$ and $G'$ are both contained in the hyperplane defining~$F$ as a codimension-$1$ face of~$M$, we can name the regions of~$\P$ as the cycle $R_0,R_1,\ldots,R_{2k}=R_0$ with $R_1=R$ and $R_k=R'$.
Since $R_1$ and $R_k$ are both in~$M$, and since~$M$ is convex, we conclude that the regions $R_1,\ldots,R_k$ are all in~$M$.
Now the polygon property implies that the regions $R_{k+1},\ldots,R_{2k}$ are all in the same maximal face of~$\C'$.
But $R_{2k}\subset N$ because $G\subset N$.
Thus in particular $R_{k+1}\subset N$, so that $G'\subset N$.
By hypothesis, $G'\subset F\subset M$, so $G'\subset M\cap N$.
We have established the claim.

Given any two codimension-$1$ faces~$G$ and $G'$ of $\C(\A)$ contained in~$F$, choosing generic points $x\in G$ and $x'\in G'$, the line segment $\overline{xx'}$ defines a sequence $G=G_0,G_1,\ldots,G_k=G'$ of codimension-$1$ faces of $\C(\A)$, contained in~$F$, such that $G_{i-1}$ and $G_i$ share a codimension-$2$ face for each $i=1,\ldots,k$.
Since there exists a codimension-$1$ face of $\C(\A)$ in $M\cap N$, the claim and a simple induction on~$k$ establish that $F\subseteq M\cap N$.
Thus $F=M\cap N$, so that $M\cap N$ is a face of~$M$.
By symmetry, $M\cap N$ is a face of~$N$ as well.
\end{proof}

Suppose~$\C$ is a pure, full-dimensional, locally strongly connected subcomplex of $\C(\A)$, and suppose $\E$ is a set of edges of~$\G$ having the polygon property.
Lemma~\ref{weak poly convex} says that $\E$ is the edge set of a pure, full-dimensional pre-complex $\C'$ coarsening $\C$ and that the set $\M$ of maximal faces of $\C'$ consists of $n$-dimensional polytopes with pairwise disjoint interiors.
Lemma~\ref{codim 1 intersections} says that $M\cap N$ is a face of~$M$ and of~$N$ for all pairs $M,N\in\M$ such that $\dim(M\cap N)>n-2$.
We now show that $\M$ satisfies condition (i) of Theorem~\ref{shortcut} with $k=n-2$.
Let $x$ be a point in $\Supp(\M)=\Supp(\C)$ and let $F$ be the face of $\C$ such that $x\in\relint(F)$.
Since $\C$ is locally strongly connected, by considering the local adjacency graph of $\C$ at $F$, we see that there exists $\ep>0$ such that $\Supp(\M)\cap\B_\delta(x)$ has a connected interior for all $\delta$ with $\ep>\delta>0$.
Then since $\bigcap_{n-2}(\M)$ is at most $(n-2)$-dimensional, the set $\left[\Supp(\M)\setminus\bigcap_{n-2}(\M)\right]\cap\B_\delta(x)$ is path connected.
Theorem~\ref{shortcut} completes the proof of Theorem~\ref{coarsen} for this special choice of $\C$.

Now let~$\C$ be any pure, full-dimensional, locally strongly connected polyhedral complex and let~$\A$ be the arrangement of hyperplanes consisting of hyperplanes containing codimension-$1$ faces of~$\C$.
Then every polyhedron in~$\C$ is a union of faces of $\C(\A)$.
Let~$\oC$ be the subcomplex of $\C(\A)$ consisting of faces contained in $\Supp(\C)$.
Then~$\C$ coarsens~$\oC$.
It is easily verified that $\oC$ is locally strongly connected.
(Local adjacency graphs of $\C$ are obtained from local adjacency graphs of $\oC$ by contracting edges.)

Continue the notation~$\G$ for the adjacency graph on maximal faces of~$\C$, and let~$\oG$ be the adjacency graph on full-dimensional faces of $\oC$.
Let~$\E$ be a set of edges of $\G$ with the ridge property.
We now define a set~$\oE$ of edges of~$\oG$, describing~$\oE$ as a set of codimension-$1$ faces of~$\oC$.
Each such face represents the pair of regions containing it.
Let~$\oF$ be a codimension-$1$ face of~$\oC$, not contained in the boundary of $\Supp(\C)$.
If~$\oF$ is contained in a codimension-$1$ face $F$ of~$\C$, then choose~$\oF$ to be an edge in~$\oE$ if and only if $F$ is an edge in~$\E$.
Otherwise, since~$\C$ coarsens~$\oC$,~$\oF$ is contained in a full-dimensional face of~$\C$ and intersects the interior of that full-dimensional face.
In this case, choose~$\oF$ to be an edge in~$\oE$.

Let $\oF$ be a codimension-$2$ face of~$\oC$.
If~$\oF$ is contained in a codimension-$2$ face $F$ of~$\C$, then $\Aff(\oF)$ and $\Aff(F)$ coincide.
Since~$\E$ has the ridge property, it restricts to the edge set of a fan in $\Perp(F)$ which coarsens $\C|_{F}$.
But~$\oE$ defines the same fan in $\Perp(\oF)=\Perp(F)$, so~$\oE$ has the ridge property at~$\oF$.
If~$\oF$ is contained in a codimension-$1$ face $F$ of~$\C$ and intersects the relative interior of $F$, then there are two cases, depending on whether $F$ is in the interior of $\Supp(\C)$ or on the boundary of $\Supp(\C)$.
In either case, it is immediate that the restriction of $\oE$ is the edge set of a fan coarsening $\oC|_{\oF}$.
In one case, the fan has two maximal cones, which are halfplanes, and in the other case, the fan has a unique maximal cone, which is a halfplane.
Since~$\C$ coarsens~$\oC$, the only possibility remaining is that~$\oF$ is contained in, and intersects the interior of, a full-dimensional face $F$ of~$\C$.
In this case, the restriction of $\oE$ defines the fan having the whole plane as its unique nonempty face.
In all cases,~$\oE$ has the ridge property at~$\oF$.
Since~$\oF$ was chosen arbitrarily,~$\oE$ has the ridge property.

By the special case of Theorem~\ref{coarsen} already proved,~$\oE$ is the edge set of a complex~$\C'$ coarsening~$\oC$.
For each maximal face~$M$ of~$\C$, every pair of adjacent regions of~$\A$ contained in~$M$ is an edge in~$\oE$, so~$\C'$ coarsens~$\C$ as well.
We will complete the proof by showing that~$\E$ is the edge set of~$\C'$, as a coarsening of~$\C$.
Let~$M$ and~$N$ be adjacent maximal faces of~$\C$.
If~$M$---$N$ is an edge in~$\E$, then $M\cap N$ is a union of codimension-$1$ faces of~$\oC$, each of which defines an edge in~$\oE$, so~$M$---$N$ is in the edge set of~$\C'$ as a coarsening of~$\C$.
If not, then $M\cap N$ is a union of codimension-$1$ faces of~$\oC$ none of which defines an edge in~$\oE$, so~$M$---$N$ is not in the edge set of~$\C'$ as a coarsening of~$\C$.
We have proved the general case of Theorem~\ref{coarsen}.

We conclude the section by adapting the above arguments to prove a special case of Tietze's convexity theorem.
This will be generalized to oriented matroids as Theorem~\ref{Tietze OM}.
\begin{theorem}\label{Tietze lite}
Let $\M$ be a finite set of $n$-dimensional polyhedra in $\reals^n$.
Suppose:
\begin{enumerate}
\item[(i) ] The interior of $\Supp(\M)$ is path-connected; and
\item[(ii) ] For every~$x$ in the boundary of $\Supp(\M)$, there exists a closed halfspace $H^+$ bounded by a hyperplane $H$ such that $x\in H$ and $\bigcup_{x\in M\in\M}M\subseteq H^+$.
\end{enumerate}
Then~$\Supp(\M)$ is convex.
\end{theorem}
\begin{proof}
Let~$\A$ be the set of hyperplanes that contain codimension-$1$ faces of polyhedra in~$\M$.
Let~$\oM$ be the set of regions of~$\A$ contained in $\Supp(\M)$.
Let~$\oE$ be the set of edges $Q$---$R$ in $\G(\A)$ such that $Q$ and $R$ are both in~$\oM$.
Let~$U$ be the union of all faces of codimension at least $2$ of polyhedra of~$\oM$.
Since the interior of $\Supp(\M)=\Supp(\oM)$ is path-connected and full-dimensional, for any two polyhedra $M$ and $N$ in~$\oM$, there exists a continuous path in $\int(\Supp(\oM))\setminus U$ from $\int(M)$ to $\int(N)$.
This implies that $\oM$ is a connected component of~$\oE$.

Let $\P$ be a $2k$-gon in $\G(\A)$ defined by a codimension-$2$ face~$F$ of $\C(\A)$.
Suppose that $k$ consecutive edges of $\P$ are in~$\oE$ but that not all of the edges of $\P$ are in~$\oE$.
Then any point~$x$ in the relative interior of~$F$ provides a violation of hypothesis~(ii).
Thus~$\oE$ has the weak polygon property, so~$\oE$ defines a pre-complex $\C'$ by Lemma~\ref{weak poly convex} (with $\C=\C(\A)$).
Since $\oM$ is a connected component of~$\oE$, one of the full-dimensional polyhedra in $\C'$ is $\Supp(\oM)=\Supp(\M)$.
\end{proof}

\section{Polyhedral complexes in oriented matroids}\label{OM sec}
In this section, we extend our results to the context of oriented matroids.
We base our approach to oriented matroids on \cite[Chapter~4]{OrientedMatroids}.

Let $\L\subseteq\set{+,-,0}^E$ be the set of covectors of an oriented matroid with no loops, over a finite ground set $E$.
The symbol~$\L$ will also denote the corresponding (``big'') face lattice.
A \newword{closed halfspace} in~$\L$ is a set $H^+_e$ consisting of all covectors in~$\L$ having component $+$ or $0$ in position $e$, or a set $H^-_e$ consisting of all covectors having component $-$ or $0$ in position $e$.
(In \cite[Chapter~4.2]{OrientedMatroids}, a different notion of halfspaces appears, but translating between the two conventions is easy.)
A \newword{hyperplane} in~$\L$ is a set $H_e=H_e^+\cap H_e^-$.
A \newword{polyhedron} $P$ in~$\L$ is an intersection of closed halfspaces.
The \newword{rank} of a polyhedron $P$ is the rank of the maximal covectors in $P$.
An \newword{exposed face} of $P$ is a subset of $P$ of the form $P\cap H_e$ for any $e\in E$ with $P\subseteq H^+_e$ or $P\subseteq H^-_e$.
A \newword{face} of $P$ is any intersection of exposed faces, including the empty intersection, which is interpreted as $P$.
A set $U$ of covectors is \newword{closed} if it is an order ideal in $\L$.
The \newword{boundary} of a closed set~$U$ of covectors is the set of covectors~$X$ in~$U$ such that the principal order filter of~$\L$ generated by~$X$ is not contained in~$U$.
The \newword{interior} of~$U$ is~$U$ minus the boundary of~$U$.
A set of covectors is \newword{connected} if it induces a connected subgraph of the Hasse diagram of~$\L$.

The definition of a \newword{polyhedral complex} in~$\L$ and terminology for complexes is copied verbatim from Section~\ref{intro}, with $\reals^n$ replaced by~$\L$.

\begin{theorem}\label{shortcut OM}
Fix $k\ge -1$ and let~$\M$ be a finite collection of polyhedra in an oriented matroid~$\L$, each of rank greater than~$k$.
Suppose:
\begin{enumerate}
\item[(i) ] For each $X\in\Supp(\M)$, the set $\set{Y\in\Supp(\M)\setminus\bigcap_k(\M):Y\ge X}$ is connected.
(Here $\bigcap_k(\M)$ is defined as in Section~\ref{intro}.)
\item[(ii) ] $M\cap N$ is a face of~$M$ and of~$N$ for all $M,N\in\M$ with $\rank(M\cap N)>k$.
\end{enumerate}
Then the collection of all polyhedra in~$\M$ and their faces is a polyhedral complex.
\end{theorem}
\begin{proof}
The oriented matroid version of Lemma~\ref{longer shortcut} holds by a proof that is the same, except for a modification of the argument that $F\cap N$ is a face of $M\cap N$.
For each exposed face $F'$ of $M$ containing $F$, we argue as in the proof of Lemma~\ref{longer shortcut} that $F'\cap N$ is an exposed face of $M\cap N$.
Since $F$ is the intersection of all exposed faces of $M$ containing it, we conclude that $F\cap N$ is an intersection of exposed faces of $M\cap N$, or in other words that $F\cap N$ is a face of $M\cap N$.

Let~$M$ and~$N$ be distinct polyhedra in~$\M$, let~$F$ be the polyhedron $M\cap N$, and let $d$ be the rank of~$F$.
Let~$X$ be a rank-$d$ covector in~$F$.
Then the covectors in $\Supp(\M)\setminus\bigcap_k(\M)$ that are above~$X$ in the lattice~$\L$ form a connected set.
Let~$Y$ be a covector of full rank in~$M$ and~$Z$ a covector of full rank in $N$.
Then there is a path $Y=X_0,X_1,\ldots,X_j=Y$ in $\Supp(\M)\setminus\bigcap_k(\M)$ consisting of covectors above~$X$.
Let $M_0=M$ and $M_j=N$ and for each $i=1,\ldots,j-1$, let $M_i$ be any polyhedron in $\M$ with $X_i\in M_i$.
The result is a sequence $M=M_0,\ldots,M_j=N$ of polyhedra in $\M$, all containing~$X$, such that, for each $i=1,\ldots,j$, the intersection $M_{i-1}\cap M_i$ is of rank greater than~$k$.
As in the proof of Theorem~\ref{shortcut}, we see that $F'=M_0\cap\cdots\cap M_j$ is a face of~$N$.
By construction, $F'$ contains~$X$ and so has rank at least~$d$.
Since $F'\subseteq F$, it has rank exactly~$d$.

Let $G=N\cap H$ be any exposed face of $N$ containing $F'$.
Then $F'\subseteq H$.
Since $F$ and $F'$ are rank-$d$ polyhedra with $F'\subseteq F$, we have $F\subseteq H$ and so $F\subseteq G$.
Thus $F\subseteq F'$, so~$F=F'$ is a face of~$N$.
By symmetry,~$F$ is a face of~$M$.
\end{proof}

If the rank of $\L$ is $n$, then $\set{X\in\L:\rank(X)\in\set{n-1,n}}$ is connected by \cite[Proposition~4.2.3]{OrientedMatroids}.
The following proposition is \cite[Proposition~4.2.6]{OrientedMatroids}.

\begin{prop}\label{OM prop}
Let $P$ be an order ideal in~$\L$ generated by a set of covectors of full rank $n$.
Then $P$ is a polyhedron if and only if the following condition holds:
If $Y$ and $Z$ are maximal in $P$ and $X$ is on a shortest path from $X$ to $Y$ in $\Set{X\in \L:\rank(X)\in\set{n-1,n}}$, then $X\in P$.
\end{prop}

Paths and path convexity in~$\L$ are defined as in the realizable case.
Replacing Theorem~\ref{red path prop} with \cite[Theorem~2.4]{CorMor} and establishing, by the same proof, the analog of Lemma~\ref{red path enough}, we obtain the following theorem.
(Cf. \cite[Proposition~4.4.6]{OrientedMatroids}.)

\begin{theorem}\label{path conv OM}
Every oriented matroid is path convex.
\end{theorem}

Let $F$ be a corank-$2$ face of a polyhedral complex $\C$ in $\L$.
Define $\Perp(F)$ to be the rank-$2$ oriented matroid obtained by the deletion from $\L$ of all $e\in E$ such that $H_e$ does not contain~$F$.
Since every rank-$2$ oriented matroid is realizable, we can think of $\Perp(F)$ as a fan.
For each $M\in\M$ with $F\subset M$, define a cone in $\Perp(F)$ as the set of covectors in $\Perp(F)$ which arise by deletion from covectors in $M$.
The fan $\C|_F$ is the collection of all such cones.
The ridge property on a set $\E$ of edges in the adjacency graph $\G$ is now defined exactly as in the realizable case.

\begin{theorem}\label{coarsen OM}
Let~$\C$ be a pure, full-rank, locally strongly connected polyhedral complex in an oriented matroid~$\L$.
Then a subset~$\E$ of the edges of the adjacency graph~$\G$ is the edge set of a complex coarsening~$\C$ if and only if~$\E$ has the ridge property.
\end{theorem}

The proof of Theorem~\ref{coarsen OM} runs along the same lines as the proof of Theorem~\ref{coarsen}.
First, the oriented-matroid analog of Lemma~\ref{weak poly convex} holds by essentially the same proof, but instead of considering a line segment between points in $M$, we appeal to Proposition~\ref{OM prop}.
The analog of Lemma~\ref{codim 1 intersections} also holds by the same proof, except that instead of using a line segment $\overline{xx'}$ to construct a sequence of corank-$1$ covectors contained in $F$, we appeal to Proposition~\ref{OM prop}, applied to the restriction of~$\L$ to the hyperplane containing $F$.

Let~$\C$ be a pure, full-rank, locally strongly connected polyhedral complex in~$\L$.
Given a covector $X$, the notation $\oX$ will refer to the polyhedron which is the principal order ideal below $X$ in~$\L$.
Thus $\oX$ should be thought of as the closure of~$X$, while~$X$ should be thought of as the relative interior of $\oX$.
Let $\C(\L)$ be the polyhedral complex in~$\L$ consisting of all polyhedra $\oX$ such that~$X$ is a covector in~$\L$.
Let~$\oC$ be the subcomplex of $\C(\L)$ consisting of faces contained in $\Supp(\C)$.
Then~$\C$ coarsens~$\oC$. 
As in the realizable case, we begin with a set $\E$ of edges in the adjacency graph of $\C$, with the ridge property, and construct a set~$\oE$ of edges of the adjacency graph of~$\oC$.
We then argue, along the same lines, that~$\oE$ has the ridge property.
Using the oriented-matroid analogs of Lemmas~\ref{weak poly convex} and~\ref{codim 1 intersections} as well as Proposition~\ref{OM prop}, we verify the hypotheses of Theorem~\ref{shortcut OM} and conclude that~$\oE$ is the edge set of a complex~$\C'$ coarsening~$\oC$.
The argument that $\E$ is the edge set of $\C'$ as a coarsening of $\C$ extends to the oriented matroid case, and this completes the proof of Theorem~\ref{coarsen OM}.

Echoing the proof of Theorem~\ref{Tietze lite}, with $\C(\A)$ replaced by $\C(\L)$, we obtain an oriented matroid version of Tietze's convexity theorem.

\begin{theorem}\label{Tietze OM}
Let $\M$ be a finite set of polyhedra of full rank in~$\L$.
Suppose:
\begin{enumerate}
\item[(i) ] The interior of $\Supp(\M)$ is connected; and
\item[(ii) ] For all~$X$ in the boundary of $\Supp(\M)$, there exists a closed halfspace $H^+$ bounded by a hyperplane $H$ such that $X\in H$ and $\bigcup_{x\in M\in\M}M\subseteq H^+$.
\end{enumerate}
Then~$\Supp(\M)$ is a polyhedron in~$\L$.
\end{theorem}

\begin{remark}
When $\L$ is realizable, polyhedral complexes in $\L$, as defined above, are \emph{fans} rather than polyhedral complexes.
One can similarly define polyhedral complexes and prove all of the results of this section in the context of \newword{affine oriented matroids} as in \cite[Section~4.5]{OrientedMatroids}.
As usual, the decision to pass from the linear to the affine or vice-versa is a matter of convenience, not a matter of mathematical substance.
\end{remark}

\section*{Acknowledgments}
The author gratefully acknowledges the contributions of several colleagues:
Bernd Sturmfels made the author aware of the problem of generalizing \cite[Theorem~9]{ranktests} from the symmetric group to arbitrary finite Coxeter groups and contributed helpful comments.
Ezra Miller suggested that the argument for Theorem~\ref{coarsen} should imply some local condition for convexity (see Theorem~\ref{Tietze lite}).
Sergei Ivanov answered the author's question on \texttt{mathoverflow.net}, making the author aware of Tietze's convexity theorem and giving a reference.
Isabella Novik, Igor Pak, Vic Reiner, and an anonymous referee contributed helpful comments and references.
In an earlier version of this paper, Theorems~\ref{coarsen} and~\ref{coarsen OM} were proved under the stronger hypothesis that $\C$ has convex support.
The referee pointed out that the argument goes through, with only trivial modifications, for $\C$ pure, full-dimensional, and locally strongly connected.


\begin{thebibliography}{9}

\bibitem{OrientedMatroids}
A. Bj\"{o}rner, M. Las Vergnas, B. Sturmfels, N. White and G. Ziegler,
\textit{Oriented matroids}
(Second edition),
 Encyclopedia of Mathematics and its Applications, {\bf 46}, 
Cambridge Univ. Press, 1999.

\bibitem{CorMor}
R. Cordovil and M. L. Moreira,
\textit{A homotopy theorem on oriented matroids.}
Discrete Math. \textbf{111} (1993), no. 1--3, 131--136. 

\bibitem{Deligne}
P. Deligne,
\textit{Les immeubles des groupes de tresses g\'{e}n\'{e}ralis\'{e}s.}
Invent. Math. \textbf{17} (1972), 273--302. 

\bibitem{ranktests}
J. Morton, L. Pachter, A. Shiu, B. Sturmfels and O. Wienand,
\textit{Convex rank tests and semigraphoids.}
SIAM J. Discrete Math. \textbf{23} (2009), no. 3, 1117--1134. 

\bibitem{con_app}
N.~Reading, 
\textit{Lattice congruences, fans and Hopf algebras.}
J. Combin. Theory Ser. A \textbf{110} (2005), no. 2, 237--273. 

\bibitem{Salvetti}
M. Salvetti,
\textit{Topology of the complement of real hyperplanes in $C^N$.}
Invent. Math. \textbf{88} (1987), no. 3, 603--618. 

\bibitem{Tietze}
H. Tietze,
\textit{Bemerkungen \"{u}ber konvexe und nicht-konvexe Figuren.}
J. Reine Angew. Math. \textbf{160} (1929), 67--69.

\bibitem{Tits}
J. Tits,
\textit{Le probl\`{e}me des mots dans les groupes de Coxeter.}
1969 Symposia Mathematica (INDAM, Rome, 1967/68), Vol. \textbf{1}, 175--185.
Academic Press, London.

\bibitem{Valentine}
F. A. Valentine,
\textit{Convex sets.}
McGraw-Hill Book Co., New York-Toronto-London, 1964.

\bibitem{Ziegler}
G. Ziegler,
\textit{Lectures on polytopes.}
Graduate Texts in Mathematics, \textbf{152}. 
Springer-Verlag, New York, 1995.


\end{thebibliography}
\end{document}